\documentclass{amsart}

\usepackage{amsthm}
\usepackage{graphicx}
\usepackage{stmaryrd} 
\usepackage{tikz}
\usepackage{tikz-cd}
\usepackage{ bbold }
\usetikzlibrary{arrows}
\usepackage{verbatim}
\usepackage{amssymb,amsfonts}
\usepackage[all,arc]{xy}
\usepackage{enumerate}
\usepackage{mathrsfs, mathabx}
\usepackage{xcolor}[2007/01/21]
\usepackage{hyperref}
\usepackage{soul}
\usepackage{esint} 
\usepackage[all]{xy}
\usepackage[margin=1.0in]{geometry}
\usepackage[
   open,
   openlevel=2,
   atend
 ]{bookmark}[2011/12/02]
\usepackage{graphics}

\bookmarksetup{color=blue}

\theoremstyle{definition} 
\newtheorem{thm}{Theorem}[section]
\newtheorem{cor}[thm]{Corollary}
\newtheorem{prop}[thm]{Proposition}
\newtheorem{lem}[thm]{Lemma}

\theoremstyle{definition}

\newtheorem{rmk}[thm]{Remark}

\theoremstyle{definition}
\newtheorem*{ax1}{Axiom 1}
\newtheorem*{ax2}{Axiom 2}
\newtheorem*{ax3}{Axiom 3}

\theoremstyle{remark}

\newcommand{\GL}{\mathrm{GL}}

\newcommand{\Tr}{\mathrm{Tr}}

\newcommand{\End}{\mathrm{End}}
\newcommand{\bZ}{\mathbb{Z}}
\newcommand{\bQ}{\mathbb{Q}}

\newcommand{\bP}{\mathbb{P}}
\newcommand{\bA}{\mathbb{A}}
\newcommand{\bC}{\mathbb{C}}
\newcommand{\bF}{\mathbb{F}}

\newcommand{\mr}{\mathrm}

\newcommand{\mc}{\mathcal}

\newcommand{\Spec}{\mathrm{Spec}}

\newcommand{\Sym}{\mathrm{Sym}}
\newcommand{\Alt}{\mathrm{Alt}}

\newcommand{\Mat}{\mathrm{Mat}}

\newcommand{\Fr}{\mathrm{Fr}}

\newcommand{\llb}{\llbracket}
\newcommand{\rrb}{\rrbracket} 
 
\newcommand{\GrVec}{\textbf{GrVec}}

\newcommand{\eps}{\epsilon}

\newcommand{\ld}{\lambda}

\newcommand{\id}{\mathrm{id}}
\newcommand{\bs}{\boldsymbol}
\newcommand{\ol}{\overline}

\newcommand{\bl}{\bullet}

\newcommand{\ra}{\rightarrow}

\newcommand{\hra}{\hookrightarrow}

\newcommand{\lt}{\left}
\newcommand{\rt}{\right}

\newcommand{\leqs}{\leqslant}

\newcommand{\ot}{\otimes}
\newcommand{\op}{\oplus}
\newcommand{\acts}{\lefttorightarrow}
\newcommand{\bop}{\bigoplus}

\newcommand{\opp}{\text{op}}

\newcommand{\et}{\text{\'et}}

\newcommand{\be}{\begin{enumerate}}
\newcommand{\ee}{\end{enumerate}}

\numberwithin{equation}{section}

\begin{document}

\title[P\'olya enumeration theorems in algebraic geometry]{P\'olya enumeration theorems in algebraic geometry}
\author{Gilyoung Cheong}
\address{Department of Mathematics, University of Michigan, 530 Church Street, Ann Arbor, MI 48109-1043}
\email{gcheong@umich.edu}

\begin{abstract}
We generalize a formula due to Macdonald that relates the singular Betti numbers of $X^{n}/G$ to those of $X$, where $X$ is a compact manifold and $G$ is any subgroup of the symmetric group $S_{n}$ acting on $X^{n}$ by permuting coordinates. Our result is completely axiomatic: in a general setting, given an endomorphism on the cohomology $H^{\bl}(X)$, it explains how we can explicitly relate the Lefschetz series of the induced endomorphism on $H^{\bl}(X^{n})^{G}$ to that of the given endomorphism on $H^{\bl}(X)$ in the presence of the K\"unneth formula with respect to a cup product. For example, when $X$ is a compact manifold, we take the Lefschetz series given by the singular cohomology with rational coefficients. On the other hand, when $X$ is a projective variety over a finite field $\bF_{q}$, we use the $l$-adic \'etale cohomology with a suitable choice of prime number $l$. We also explain how our formula generalizes the P\'olya enumeration theorem, a classical theorem in combinatorics that counts colorings of a graph up to given symmetries, where $X$ is taken to be a finite set of colors. When $X$ is a smooth projective variety over $\bC$, our formula also generalizes a result of Cheah that relates the Hodge numbers of $X^{n}/G$ to those of $X$. We will also see that the generating function for the Lefschetz series of the endomorphisms on $H^{\bl}(X^{n})^{S_{n}}$ is rational, and this generalizes the following facts: 1. the generating function of the Poincar\'e polynomials of symmetric powers of a compact manifold $X$ is rational; 2. the generating function of the Hodge-Deligne polynomials of symmetric powers of a smooth projective variety $X$ over $\bC$ is rational; 3. the zeta series of a projective variety $X$ over $\bF_{q}$ is rational. We also prove analogous rationality results when we replace $S_{n}$ with $A_{n}$, alternating groups.
\end{abstract}

\maketitle

\section{Introduction} 

\subsection{Motivation} Let $X$ be a compact complex manifold of (complex) dimension $d$ and consider the $n$-th symmetric power $\Sym^{n}(X) = X^{n}/S_{n}$ for each $n \in \bZ_{\geq 0}$. One may ask how to compute the singular Betti numbers $h^{0}(\Sym^{n}(X)), h^{1}(\Sym^{n}(X)), \dots$ for various $n$ with respect to those of $X$. In his influential paper \cite{Mac1}, Macdonald settled this question: he proved

$$\sum_{n=0}^{\infty}\chi_{u}(\Sym^{n}(X)) t^{n} = \frac{(1 - ut)^{h^{1}(X)} \cdots (1 - u^{2d-1}t)^{h^{2d-1}(X)}}{(1 - t)^{h^{0}(X)} \cdots (1 - u^{2d}t)^{h^{2d}(X)}},$$

\

where

$$\chi_{u}(Y) := \sum_{i=0}^{\infty}(-u)^{i}h^{i}(Y),$$

\

a power series\footnote{In this paper, we call $\chi_{u}(Y)$ the \textbf{Poincar\'e series} of $Y$ although it is more common to use the terminology for $\chi_{(-u)}(Y)$, the generating function for $h^{i}(Y)$. If $h^{i}(Y) = 0$ for large enough $i$, we have $\chi_{1}(Y) = \chi(Y)$, the Euler characteristic of $Y$.} in $u$ with integral coefficients, defined for any topological space $Y$ with finite singular Betti numbers. Note that the right-hand side of the above identity is rational in $t$.

\

\hspace{3mm} There is an analogous result when $X$ is a projective variety of dimension $d$ over a finite field $\bF_{q}$ due to Grothendieck:

$$Z_{X}(t) = \frac{\det(\id_{H^{1}(X)} - \Fr_{q,1}^{*}t) \cdots \det(\id_{H^{2d-1}(X)} - \Fr_{q,2d-1}^{*}t)}{\det(\id_{H^{0}(X)} - \Fr_{q,0}^{*}t) \cdots \det(\id_{H^{2d}(X)} - \Fr_{q,2d}^{*}t)},$$

where

\

$$Z_{X}(t) = \exp\left(\sum_{r=1}^{\infty}\frac{|X(\bF_{q^{r}})|t^{r}}{r}\right)  = \prod_{x \in |X|}\frac{1}{1 - t^{\deg(x)}},$$

\

is the \textbf{zeta series} of $X$, writing $|X|$ to mean the set of closed points of $X$ in the last expression, which reminds us its similarity to the Euler product of the Riemann zeta function. In the above result, the notation $H^{i}(X)$ now denotes the $i$-th $l$-adic \'etale cohomology 

$$H^{i}_{\et}(X, \bQ_{l}) := H^{i}_{\et}(X_{/\ol{\bF_{q}}}, \bZ_{l}) \ot_{\bZ_{l}} \bQ_{l}$$

\ 

of $X_{/\ol{\bF_{q}}} := X \times_{\Spec(\bF_{q})} \Spec(\ol{\bF_{q}}),$ which is a finite dimensional vector space over the field $\bQ_{l}$ of $l$-adic rational numbers for any fixed prime number $l$ not dividing $q$. We write $\Fr_{q}$, which we call the \textbf{Frobenius} endomorphism on $X$, to mean the map from $X$ to itself given by the identity on the underlying topological space and the $q$-th power map on the structure sheaf $\mathscr{O}_{X}$, giving an endomorphism on $X_{/\ol{\bF_{q}}}$, inducing the $\bQ_{l}$-linear endomorphism $\Fr_{q,i}^{*}$ on $H^{i}(X)$. In particular, this shows that $Z_{X}(t)$ is rational in $t$, which was first shown by Dwork \cite{Dwo}. It is well-known (presumably due to Kapranov \cite{Kap}) that

$$Z_{X}(t) = \sum_{n=0}^{\infty} |\Sym^{n}(X)(\bF_{q})| t^{n},$$

\

so writing

$$\sum_{n=0}^{\infty} |\Sym^{n}(X)(\bF_{q})| t^{n} = \frac{\det(\id_{H^{1}(X)} - \Fr_{q,1}^{*}t) \cdots \det(\id_{H^{2d-1}(X)} - \Fr_{q,2d-1}^{*}t)}{\det(\id_{H^{0}(X)} - \Fr_{q,0}^{*}t) \cdots \det(\id_{H^{2d}(X)} - \Fr_{q,2d}^{*}t)},$$

\

one may visibly find the similarity between Grothendieck's formula and Macdonald's formula. When we take $u = 1$ in Macdonald's formula, we have

$$\sum_{n=0}^{\infty}\chi(\Sym^{n}(X)) t^{n} = \left(\frac{1}{1 - t}\right)^{\chi(X)},$$

\

and by making analogies between taking the Euler characteristic and counting $\bF_{q}$-points, Vakil \cite{Vak} explained how to interpret Grothendieck's formula as the specialization $u = 1$ of Macdonald's formula in the $l$-adic setting. In this paper, we take this analogy one step further by generalizing both Macdonald's formula and Grothendieck's formula.

\

\hspace{3mm} Our main theorem (i.e., Theorem \ref{main}) is too formal to state without providing a concrete consequence:

\begin{thm}\label{showcase} Let $X$ be either a compact complex manifold of dimension $d$ or a projective variety of dimension $d$ over a finite field $\bF_{q}$. Then for any endomorphism $F$ on $X$, we have

$$\sum_{n=0}^{\infty}L_{u}(\Sym^{n}(F)^{*}) t^{n} = \frac{\det(\id_{H^{1}(X)} - F_{1}^{*}ut) \cdots \det(\id_{H^{2d-1}(X)} - F_{2d-1}^{*}u^{2d-1}t)}{\det(\id_{H^{0}(X)} - F_{0}^{*}t) \cdots \det(\id_{H^{2d}(X)} - F_{2d}^{*}u^{2d}t)},$$

\

where

\begin{itemize}
	\item $H^{i}(X)$ is the singular cohomology of $X$ with $\bQ$-coefficients when $X$ is a compact complex manifold,
	\item $H^{i}(X)$ is the \'etale cohomology of $X_{/\ol{\bF_{q}}}$ with $\bQ_{l}$-coefficients when $X$ is a projective variety over $\bF_{q}$ for some prime number $l$,
	\item $\Sym^{n}(F)$ is the endomorphism on $\Sym^{n}(X)$ induced by $F$,
	\item $F_{i}^{*}$ is the induced endomorphism on $H^{i}(X)$ from $F$, and
	\item $L_{u}(F^{*}) := \sum_{i \geq 0}(-u)^{i}\Tr(F_{i}^{*})$.
\end{itemize}
\end{thm}

\

\hspace{3mm} Indeed, taking $F = \id_{X}$ in the singular setting of Theorem \ref{showcase}, we obtain Macdonald's formula. Taking $F = \Fr_{q}$ in the $l$-adic setting for choosing primes $l \nmid q$ with $u = 1$, we obtain Grothendieck's formula thanks to the Grothendieck-Lefschetz trace formula (e.g., \cite{Mil}, VI, Theorem 13.4), which implies that

$$|\Sym^{n}(X)(\bF_{q})| = L_{1}(\Sym^{n}(\Fr_{q})^{*}),$$

\

noting that $\Sym^{n}(\Fr_{q})$ is equal to the Frobenius endomorphism on $\Sym^{n}(X)$. Note that Theorem \ref{showcase} is more general than the two formulas in either setting, and we still get the rational generating function in $t$. Our main theorem, which we introduce in the next subsection, is much more general than Theorem \ref{showcase}, and yet it is a simple representation-theoretic observation. We hope that experts in various cohomology theories may find our general formulation clear and useful.

\

\begin{rmk} Theorem \ref{showcase} holds more generally, although we do not seek its maximum generality in this paper. For the singular setting, one may take $X$ to be any compact smooth manifold of (real) dimension $2d$ or any finite CW complex such that $h^{i}(X) = 0$ for all $i > 2d$. In the $l$-adic setting, we must require $l > n$ whenever we deal with $H^{i}(\Sym^{n}(X)) = H^{i}_{\et}(\Sym^{n}(X), \bQ_{l})$ because the result depends on the isomorphism $H^{i}(\Sym^{n}(X)) \simeq H^{i}(X^{n})^{S_{n}}$ (\cite{HN}, Proposition 3.2.1) that uses the fact that $l$ does not divide $|S_{n}| = n!$. Moreover, since $H^{i}(X) = H^{i}_{\et}(X, \bQ_{l}) = 0$ for $i > 2d$ (\cite{Mil}, VI, Theorem 1.1), we have

$$H^{i}(\Sym^{n}(X)) \simeq H^{i}(X^{n})^{S_{n}} \simeq \left(\bop_{i_{1} + \cdots + i_{n} = i} H^{i_{1}}(X) \ot \cdots \ot H^{i_{n}}(X) \right)^{S_{n}} = 0$$

\

if $i > 2dn$ with any choice of $l > n$, so each $L_{u}(\Sym^{n}(F)^{*})$ is a polynomial in $u$ for any such $l$. In the $l$-adic setting, one can instead use the compactly supported $l$-adic \'etale cohomology $H^{\bl}_{\et, \mr{c}}(X, \bQ_{l})$ instead, which allows us to consider $X$ to be any quasi-projective variety over $\bF_{q}$. This will be particularly interesting for revisiting a previously known point-counting result over $\bF_{q}$ in Section \ref{pointcount}.
\end{rmk}

\

\hspace{3mm} Later in this paper, we will run the same story, replacing the full symmetric groups $S_{n}$ with their alternating subgroups $A_{n}$. In particular, we will obtain the following analogue of Theorem \ref{showcase}. This will be restated as Theorem \ref{showcase2}, and more concrete consequences of this can be found in Section \ref{alt}:

\begin{thm} Let $X$ be either a compact complex manifold of dimension $d$ or a projective variety of dimension $d$ over a finite field $\bF_{q}$. Then for any endomorphism $F$ on $X$, we have

$$\sum_{n=0}^{\infty}L_{u}(\Alt^{n}(F)^{*}) t^{n} = \prod_{i=0}^{2d} \left( \frac{1}{\det(\id_{H^{i}(X)} - F^{*}_{i}u^{i}t)} \right)^{(-1)^{i}} +  \prod_{i=0}^{2d} \left( \frac{1}{\det(\id_{H^{i}(X)} + F^{*}_{i}u^{i}t)} \right)^{(-1)^{i+1}} - 1 - L_{u}(F^{*}),$$

\

where we used the same notations as in Theorem \ref{showcase} except $\Alt^{n}(F)$, the endomorphism on the $n$-th alternating power $\Alt^{n}(X) = X^{n}/A_{n}$ of $X$ induced by $F$.
\end{thm}

\

\subsection{Main result and its applications}\label{setup} In this subsection, we formulate our main result. Let $\mc{C}$ be a category where any finite products exist. Fix a field $k$, and suppose that we have a functor 

$$H^{\bl} : \mc{C}^{\opp} \ra \GrVec_{k}$$

\

from the opposite category $\mc{C}^{\opp}$ of $\mc{C}$ to the category $\GrVec_{k}$ of $\bZ_{\geq 0}$-graded vector spaces over $k$ whose morphisms are $k$-linear graded maps (of degree $0$). Given any object $X$ in $\mc{C}$, we may write

$$H^{\bl}(X) = \bop_{i=0}^{\infty}H^{i}(X),$$

\

where each $H^{i}(X)$ is a vector space over $k$. Given any morphism $f : X \rightarrow Y$ in $\mc{C}$, the induced $k$-linear map $f^{*} : H^{\bl}(Y) \rightarrow H^{\bl}(X)$ can be decomposed into $f_{i}^{*} : H^{i}(Y) \rightarrow H^{i}(X)$ for each $i \in \bZ_{\geq 0}$ by definition. In addition, we assume the following axioms:

\

\begin{ax1}
Given any object $X$ in $\mc{C}$, we assume that there is a \textbf{cup product}, namely a $k$-bilinear map $\cup : H^{i}(X) \times H^{j}(X) \rightarrow H^{i+j}(X)$ defined for each $i, j \in \bZ_{\geq 0}$ such that

$$a \cup b = (-1)^{ij} b \cup a$$

\

for all $a \in H^{i}(X)$ and $b \in H^{j}(X)$. 
\end{ax1}

\

\begin{ax2}
Assuming Axiom 1, given any objects $X$ and $Y$ in $\mc{C}$, we assume the \textbf{K\"unneth formula}:

$$H^{\bl}(X \times Y) \simeq H^{\bl}(X) \ot_{k} H^{\bl}(Y)$$

\

given by $p_{X}^{*}(a) \cup p_{Y}^{*}(b) \mapsto a \ot b$ for any \textbf{homogeneous} elements $a \in H^{\bl}(X)$ and $b \in H^{\bl}(Y)$, meaning $a \in H^{i}(X)$ and $b \in H^{j}(Y)$ for some $i, j \in \bZ_{\geq 0}$. (In this case, we will write $i = \deg(a)$ and $j = \deg(b)$ and call them the \textbf{degree} of $a$ and that of $b$, respectively, for the rest of this paper.)
\end{ax2}

\

\begin{ax3} Given any object $X$ in $\mc{C}$, the $k$-vector space $H^{i}(X)$ is finite-dimensional.
\end{ax3}

\

\hspace{3mm} The reader may immediately note that Axiom 1 is only meaningful due to Axiom 2 since otherwise one can always give trivial bilinear maps for a cup product of $H^{\bl}(X)$. Note that these two axioms give

$$H^{\bl}(X^{n}) \simeq H^{\bl}(X)^{\ot n}$$

\

defined by

$$p_{1}^{*}(a_{1}) \cup \cdots \cup p_{n}^{*}(a_{n}) \mapsto a_{1} \ot \cdots \ot a_{n},$$

\

for any homogeneous $a_{1}, \dots, a_{n} \in H^{\bl}(X)$, where $p_{1}, \dots, p_{n}$ are the projection maps $X^{n} \ra X$. If $G$ is any subgroup of $S_{n}$, then $G$ acts on $X^{n}$ by permuting coordinates. The induced action of $G$ on $H^{\bl}(X)$ is precisely given by

$$g \cdot (p_{1}^{*}(a_{1}) \cup \cdots \cup p_{n}^{*}(a_{n})) = p_{g(1)}^{*}(a_{1}) \cup \cdots \cup p_{g(n)}^{*}(a_{n}).$$

\

for $g \in G$. If $\phi = \bop_{i=0}^{\infty}\phi_{i} : H^{\bl}(X) \rightarrow H^{\bl}(X)$ is any $k$-linear graded endomorphism, then it induces a $k$-linear graded map $\phi_{X^{n}} : H^{\bl}(X^{n}) \rightarrow H^{\bl}(X^{n})$ given by 

$$p_{1}^{*}(a_{1}) \cup \cdots \cup p_{n}^{*}(a_{n}) \mapsto p_{1}^{*}(\phi(a_{1})) \cup \cdots \cup p_{n}^{*}(\phi(a_{n})).$$

\

This map is compatible with the $G$-action we discussed above, so $\phi$ induces a $k$-linear graded map $\phi_{X^{n}}|_{H^{\bl}(X^{n})^{G}} : H^{\bl}(X^{n})^{G} \rightarrow H^{\bl}(X^{n})^{G}$ on the $G$-invariance parts. Note that if $F : X \ra X$ is an endomorphism in $\mc{C}$, then

\begin{align*}
F^{*}_{X^{n}}(p_{1}^{*}(a_{1}) \cup \cdots \cup p_{n}^{*}(a_{n})) &= p_{1}^{*}(F^{*}(a_{1})) \cup \cdots \cup p_{n}^{*}(F^{*}(a_{n})) \\
&= (F \circ p_{1})^{*}(a_{1}) \cup \cdots \cup (F \circ p_{n})^{*}(a_{n}) \\
&= (p_{1} \circ F^{n})^{*}(a_{1}) \cup \cdots \cup (p_{n} \circ F^{n})^{*}(a_{n}) \\
&= (F^{n})^{*} (p_{1}^{*}(a_{1}) \cup \cdots \cup p_{n}^{*}(a_{n})),
\end{align*}

\

so $F^{*}_{X^{n}} = (F^{n})^{*}$, where $F^{n} : X^{n} \ra X^{n}$ is induced by $F : X \ra X$. Using Axiom 3, we can define the \textbf{Lefschetz series} of $\phi$ as

$$L_{u}(\phi) := \sum_{i=0}^{\infty}(-u)^{i} \Tr(\phi_{i}) \in k \llb u \rrb.$$

\

\hspace{3mm} We are now ready to state our main theorem:

\begin{thm}\label{main} Keeping all the notations above, suppose that $H^{\bl} : \mc{C}^{\opp} \ra \GrVec_{k}$ satisfies Axiom 1, Axiom 2, and Axiom 3. If the characteristic of $k$ does not divide $|G|$, then for any object $X$ of $\mc{C}$, we have

$$L_{u}(\phi_{X^{n}}|_{H^{\bl}(X^{n})^{G}}) = Z_{G}(L_{u}(\phi), L_{u^{2}}(\phi^{2}), \dots, L_{u^{n}}(\phi^{n})),$$

\

where 

$$Z_{G}(x_{1}, \dots, x_{n}) := \frac{1}{|G|}\sum_{g \in G}x_{1}^{m_{1}(g)} \cdots x_{n}^{m_{n}(g)} \in k[x_{1}, \dots, x_{n}]$$

\

with $m_{i}(g)$ the number of $i$-cycles in the cycle decomposition of $g$ in $S_{n}$.
\end{thm}

\

\hspace{3mm} In combinatorics, the polynomial $Z_{G}(x_{1}, \dots, x_{n})$, often defined over $\bQ$, is called the \textbf{cycle index} of $G$ in $S_{n}$. Much is known about the cycle indices. For instance (e.g., from p.20 of \cite{Sta}), we have

$$\sum_{n=0}^{\infty}Z_{S_{n}}(x_{1}, \dots, x_{n}) t^{n} = \exp\left( \sum_{r=1}^{\infty}\frac{x_{r}t^{r}}{r} \right).$$

\

This immediately provides the following:

\begin{cor}\label{S_{n}} Assume the same hypotheses as in Theorem \ref{main}. If $\dim_{k}(H^{\bl}(X))$ is finite so that $H^{i}(X) = 0$ for all $i > 2d$ for some $d$, then

$$\sum_{n=0}^{\infty}L_{u}(\phi_{X^{n}}|_{H^{\bl}(X^{n})^{S_{n}}}) t^{n} = \frac{\det(\id_{H^{1}(X)} - \phi_{1}ut) \cdots \det(\id_{H^{2d-1}(X)} - \phi_{2d-1} u^{2d-1}t)}{\det(\id_{H^{0}(X)} - \phi_{0}t) \cdots \det(\id_{H^{2d}(X)} - \phi_{2d}u^{2d}t)}.$$
\end{cor}

\begin{proof} Both sides are invariant under taking any field extension of $k$, so we may assume that $k$ is algebraically closed. In particular, the field $k$ we work with is now infinite, so we may assume that $u$ is an element of $k$. By Theorem \ref{main}, We have

\begin{align*}
\sum_{n=0}^{\infty}L_{u}(\phi_{X^{n}}|_{H^{\bl}(X^{n})^{S_{n}}}) t^{n} &= \sum_{n=0}^{\infty}Z_{S_{n}}(L_{u}(\phi), L_{u^{2}}(\phi^{2}), \dots, L_{u^{n}}(\phi^{n})) t^{n} \\
&= \exp\left( \sum_{r=1}^{\infty}\frac{L_{u^{r}}(\phi^{r})t^{r}}{r} \right) \\
&= \exp\left( \sum_{r=1}^{\infty}\sum_{i=0}^{2d}\frac{(-u^{r})^{i}\Tr(\phi_{i}^{r})t^{r}}{r} \right) \\
&= \prod_{i=0}^{2d}\exp\left( \sum_{r=1}^{\infty}\frac{(-1)^{i}\Tr((\phi_{i}u^{i})^{r})t^{r}}{r} \right) \\
&= \prod_{i=0}^{2d}\exp\left( \sum_{r=1}^{\infty}\frac{\Tr((\phi_{i}u^{i})^{r})t^{r}}{r} \right)^{(-1)^{i}}.
\end{align*}

\

Hence, the result follows from the fact that 

$$\exp\left( \sum_{r=1}^{\infty}\frac{\Tr(A^{r})t^{r}}{r} \right) = \frac{1}{\det(\id - tA)}$$

\

for any linear map $A$ on a finite dimensional vector space $V$ (e.g., see \cite{Mus}, Lemma 4.12).
\end{proof}

\

\hspace{3mm} Theorem \ref{showcase} is an immediate corollary of Corollary \ref{S_{n}}. This is because, in either the singular or the $l$-adic setting, we have the quotient map $X^{n} \rightarrow X^{n}/S_{n} = \Sym^{n}(X)$ either in the category of topological spaces or the category of varieties over $\bF_{q}$, and the map induces an isomorphism

$$H^{\bl}(\Sym^{n}(X)) \simeq H^{\bl}(X^{n})^{S_{n}}$$

\

in either setting, whose proofs can be found in \cite{Mac2} and \cite{HN} (Proposition 3.2.1) as long as we choose $l > n$ in the $l$-adic setting.

\

\hspace{3mm} Over the course of proving Theorem \ref{main}, we will show that

$$L_{u}(g \phi_{X^{n}}) = L_{u}(\phi)^{m_{1}(g)}L_{u^{2}}(\phi^{2})^{m_{2}(g)} \cdots L_{u^{n}}(\phi^{n})^{m_{n}(g)} \in k\llb u \rrb$$

\

without any assumption on the characteristic of the base field $k$ for the cohomology. When $\dim_{k}(H^{\bl}(X))$ is finite (i.e., $H^{i}(X) = 0$ for $i \gg 0$), taking $u = 1$ and $\phi = \id_{H^{\bl}(X)}$ in the above identity gives us

$$\sum_{i=0}^{\infty}(-1)^{i}\Tr(g \acts H^{i}(X^{n})) = \chi(X)^{m_{1}(g) + 2m_{2}(g) + \cdots + nm_{n}(g)} = \chi(X)^{n}$$

\

for any $g \in G$, where $\chi(X) = \sum_{i \geq 0}(-1)^{i}\dim_{k}(H^{i}(X))$. In the $l$-adic setting, the expression on the left-hand side is generally known to be an integer independent to the choice of $l$, due to Deligne and Lusztig (\cite{DL}, Proposition 3.3) for $l \nmid q$. Illusie and Zheng also provides a similar result (\cite{IZ}, Corollary 2.6). If $X$ is a smooth projective variety over $\bF_{q}$, we know that $\chi(X)$ is independent to the choice of $l$ as a consequence of a theorem of Deligne which states that the size of the eigenvalues of the Frobenius action on the $i$-th $l$-adic \'etale cohomology of $X$ is $q^{i/2}$ (e.g., see \cite{Mil}, VI, Remark 12.5.(b)). It is worth to note that in our case, the number on the left-hand side is also independent of the choice of $g \in G$, which must be due to the simplicity of the group action we are dealing with. Answering this question for our specific case does not require more than merely applying the proof of Macdonald's formula in the $l$-adic setting on top of Deligne's result.

\

\hspace{3mm} If $X$ is a smooth projective variety over $\bC$ with dimension $d$, then $i$-th singular cohomlogy $H^{i}(X)$ of (the analytification of) $X$ with $\bC$-coefficients has the Hodge decomposition:

$$H^{i}(X) = \bop_{p + q = i}H^{p,q}(X).$$

\

In general, the variety $\Sym^{n}(X)$ is not smooth, but its singular cohomology still admits the Hodge decomposition as $H^{i}(\Sym^{n}(X)) \hra H^{i}(X^{n})$ so that we can use the Hodge decomposition of $H^{i}(X^{n})$. In particular, we see $H^{i}(\Sym^{n}(X))$ has a \textbf{pure Hodge structure} of \textbf{weight} $i$, in the sense of Deligne's mixed Hodge structure introduced in \cite{Del}, although we do not need this language for the sake of this paper. In this setting, if we take

$$\phi = \bop_{i \geq 0}\bop_{p+q = i}x^{p}y^{q}\id_{H^{p,q}(X)}$$

\

for fixed $x, y \in \bC$ in Corollary \ref{S_{n}}, using $H^{\bl}(\Sym^{n}(X)) \simeq H^{\bl}(X^{n})^{S_{n}}$ (over $\bC$), we have

$$\sum_{n=0}^{\infty}\sum_{i=0}^{2d}\sum_{p+q=i}h^{p,q}(\Sym^{n}(X))x^{p}y^{q}(-u)^{i} t^{n} = \prod_{i = 0}^{2d} \prod_{p + q = i} \left( \frac{1}{1 - x^{p}y^{q}u^{i} t} \right)^{(-1)^{i}h^{p,q}(X)},$$

\

where 

$$h^{p,q}(\Sym^{n}(X)) := \dim_{\bC}(H^{i}(\Sym^{n}(X)) \cap H^{p,q}(X^{n})),$$

\

whenever $p + q = i$. Since $x, y$ are arbitrary, we may treat them as formal variables, and this identity is a result of Cheah (p.119 of \cite{Che}). When we take $u = 1$, this shows that the generating function for the Hodge-Deligne polynomials of $\Sym^{n}(X)$ is rational in $t$. This generating function is hence analogous to the zeta series of a projective variety over a finite field. Moreover, the two settings for the specialization $u = 1$ can be studied at once using the motivic zeta series of a variety defined over the Grothendieck ring of varieties as explained in \cite{Vak} and \cite{VW}. However, it is unclear whether the Grothendieck ring is the right general setting to study these phenomena when we do not specialize the variable $u$.

\

\subsection{P\'olya enumeration theorems}\label{PolyaAG} Let $G$ be a subgroup of $S_{n}$. Keeping in mind that $H^{\bl}(X^{n}/G) \simeq H^{\bl}(X^{n})^{G}$ from \cite{Mac2} and Proposition 3.2.1 of \cite{HN}, if we directly apply Theorem \ref{main} without specifying $G$ to be the full symmetric group $S_{n}$, we have

$$\chi_{u}(X^{n}/G)= Z_{G}(\chi_{u}(X), \chi_{u^{2}}(X), \dots, \chi_{u^{n}}(X))$$

\

and

$$\chi_{u}(X^{n}/G, x, y) = Z_{G}(\chi_{u}(X, x, y), \chi_{u^{2}}(X, x^{2}, y^{2}), \dots, \chi_{u^{n}}(X, x^{n}, y^{n}))$$

\

where

$$\chi_{u}(Z, x, y) := \sum_{i=0}^{\infty}\sum_{p+q = i}h^{p,q}(Z)x^{p}y^{q}(-u)^{i},$$

\

and these are also results of Macdonald \cite{Mac1} and Cheah \cite{Che}. Taking $u = 1$ in the $l$-adic setting, Theorem \ref{main} also implies the following result regrding the $\bF_{q}$-point counting:

$$|(X^{n}/G)(\bF_{q})| = Z_{G}(|X(\bF_{q})|, |X(\bF_{q^{2}})|, \dots, |X(\bF_{q^{n}})|).$$

\

One can even take $X$ to be a finite set. Then giving $X$ the discrete topology, we have $\chi(X) = |X|$, and thus

$$|X^{n}/G| = Z_{G}(|X|, |X|, \dots, |X|) = \frac{1}{|G|}\sum_{g \in G}|X|^{m(g)},$$

\

where $m(g)$ is the number of cycles in the cycle decomposition of $g$ in $S_{n}$. Since $|(X^{n})^{g}| = |X|^{m(g)}$, where $(X^{n})^{g}$ is the set of elements in $X^{n}$ fixed by $g$, the last statement also follows from Burnside's lemma. This statement is a special case of the \textbf{P\'olya enumeration theorem} in combinatorics, which we discuss in Section \ref{PolyaComb}, so it makes sense to use the same name for the preceding results including Theorem \ref{main}, even though they seem to be in the realm of algebraic geometry. This was the rationale behind the title of this paper.

\

\subsection{Structure of the rest of the paper} In Section \ref{PolyaComb}, we explain a more general version of P\'olya enumeration theorem in combinatorics and show how Theorem \ref{main} generalizes this as well. In Section \ref{mainproof}, we give a proof of Theorem \ref{main}, the main theorem of this paper. In Section \ref{alt}, we show how to compute various cohomological information about the alternating powers $\Alt^{n}(X) = X^{n}/A_{n}$ of given $X$ analogous to computations for the symmetric powers $\Sym^{n}(X) = X^{n}/S_{n}$ in the introduction. In Section \ref{pointcount}, we point out that our formula for $|(X^{n}/G)(\bF_{q})|$ holds even when $X$ is a quasi-projective variety over $\bF_{q}$. We give an example to explain why this generalization is interesting.

\

\subsection{Acknowledgments} The author thanks Yifeng Huang, Mircea Musta\c{t}\u{a}, and John Stembridge for indispensable discussions about the main ideas of this paper. When it comes to dealing with \'etale cohomology, the author is thankful for conversations with Bhargav Bhatt, Shizhang Li, Emanuel Reinecke, and Ravi Vakil. For suggesting further directions after this work, the author would like to thank Bill Fulton, Kiran Kedlaya, Luc Illusie, Minhyong Kim, Yinan Nancy Wang, and Mike Zieve. Finally, the author is grateful for comments from Daniel Litt and Weizhe Zheng about the previous draft of this paper.

\

\section{P\'olya enumeration theorem in combinatorics} \label{PolyaComb}

\hspace{3mm} Let $X = \{x_{1}, \dots, x_{r}\}$ be a finite set of colors. A common problem in combinatorics is to count the number of ways to color $n$ vertices (which we write as $1, 2, \dots, n$) of a graph with colors in $X$. The graph may have symmetries, so we want to count the colorings of $n$ vertices modulo the action of the group $G$ of symmetries of the graph. This group $G$ is a subgroup of $S_{n}$, and each coloring corresponds to an element of $X^{n}/G$, which is of the form $\bs{x} = [x_{1}]^{e_{n}} \cdots [x_{r}]^{e_{n}}$, where $e_{1}, \dots, e_{r} \in \bZ_{\geq 0}$ such that $e_{1} + \cdots + e_{r} = n$ and $x_{1}, \dots, x_{r} \in X$. We shall write $e_{i} := e_{i}(\bs{x})$ to emphasize the dependence to $\bs{x}$. Given any $(k_{1}, \dots, k_{r}) \in (\bZ_{\geq 0})^{r}$ such that $\sum_{i=1}^{r} k_{i} = n$, we may write $N_{(k_{1}, \dots, k_{r})}$ to mean the number of $\bs{x} \in X^{n}/G$ such that $e_{i}(\bs{x}) = k_{i}$ for all $1 \leq i \leq r$. We note that our counting problem is equivalent to computing the following degree $n$ homogeneous polynomial:

\

$$P_{X^{n}/G}(\bs{t}) = P_{X^{n}/G}(t_{1}, \dots, t_{r}) := \sum_{\substack{(k_{1}, \dots, k_{r}) \in (\bZ_{\geq 0})^{r}, \\ k_{1} + \cdots + k_{r} = n}} N_{(k_{1}, \dots, k_{r})}t_{1}^{k_{1}} \cdots t_{r}^{k_{r}} \in \bZ[t_{1}, \dots, t_{n}].$$

\

A classical theorem of Redfield \cite{Red}, which is also independently found by P\'olya \cite{Pol}, computes the polynomial $P_{X^{n}/G}(\bs{x})$ in terms of the subgroup $G \leqs S_{n}$. This theorem is often called \textbf{P\'olya enumeration theorem}:

\

\begin{prop}[P\'olya enumeration]\label{Polya} Given the notations above, we have

$$P_{X^{n}/G}(\bs{t}) = Z_{G}(\bs{t}, \bs{t}^{2}, \dots, \bs{t}^{n}),$$

\

where $\bs{t}^{j} := t_{1}^{j} + \cdots + t_{r}^{j}$.
\end{prop}

\

\hspace{3mm} Our main theorem, Theorem \ref{main}, generalizes this classical result. Namely, we may consider $X = \{x_{1}, \dots, x_{r}\}$ as a topological space with the discrete topology and $\phi$ the diagonal matrix on the singular cohomology

$$H^{\bl}(X) = H^{0}(X) = \bQ x_{1} \op \cdots \op \bQ x_{r} = \bQ X$$

\

whose entries are given by $t_{1}, \dots, t_{r}$. We have $(\bQ X^{n})^{G} \simeq \bQ X^{n}/G$ given by $(x_{1}, \dots, x_{n}) \mapsto [x_{1}, \dots, x_{n}]$, whose inverse is given by

$$[x_{1}, \dots, x_{n}] \mapsto \frac{1}{|G|}\sum_{g \in G}(x_{g(1)}, \dots, x_{g(n)}).$$

\

Thus, we have

$$H^{\bl}(X^{n}/G) = H^{0}(X^{n}/G) = \bQ X^{n}/G \simeq (\bQ X^{n})^{G} = H^{0}(X^{n})^{G} = H^{\bl}(X^{n})^{G},$$

\

and the induced endomorphism $\phi_{X^{n}}$ satisfies $\phi_{X^{n}} : (x_{i_{1}}, \dots, x_{i_{n}}) \mapsto t_{i_{1}} \cdots t_{i_{n}}(x_{i_{1}}, \dots, x_{i_{n}}),$ so on $\bQ X^{n}/G$, it satisfies

$$[x_{1}]^{e_{1}} \cdots [x_{r}]^{e_{r}} \mapsto t_{1}^{e_{1}} \cdots t_{r}^{e_{r}}[x_{1}]^{e_{1}} \cdots [x_{r}]^{e_{r}}.$$

\

Therefore, Theorem \ref{main} with $u = 1$ implies Proposition \ref{Polya}. That is, the classical P\'olya enumeration is a special case of Theorem \ref{main}, which deals with more than degree $0$ piece of the cohomology with more diverse choices for $X$.

\

\section{Proof of main theorem} \label{mainproof}

\hspace{3mm} In this section, we prove our main theorem, Theorem \ref{main}. We will first prove a statement about a particular permutation representation on the $n$-fold tensor product of a graded vector space and then apply it to prove Theorem \ref{main}. The representation we work with is not going to be the usual permutation representation on the pure tensors, as it will involve a sign depending on the grading. The reason is that for any $G \leqs S_{n}$, we are interested in the $G$-action on

$$H^{\bl}(X^{n}) \simeq H^{\bl}(X)^{\ot n}$$

\

given by

$$g \cdot (p_{1}^{*}(a_{1}) \cup \cdots \cup p_{n}^{*}(a_{n})) = p_{g(1)}^{*}(a_{1}) \cup \cdots \cup p_{g(n)}^{*}(a_{n}),$$

\

where $a_{1}, \dots, a_{n} \in H^{\bl}(X)$ are homogeneous elements and $g \in G$, where $p_{1}, \dots, p_{n} : X^{n} \ra X$ are projection maps. For instance, if $g = (1 \ 2)$, the transposition switching $1$ and $2$, then the corresponding action of $g$ on $H^{\bl}(X)^{\ot n}$ is given by

$$g \cdot (a_{1} \ot \cdots \ot a_{n}) = (-1)^{\deg(a_{1})\deg(a_{2})}a_{2} \ot a_{1} \ot a_{3} \ot \cdots \ot a_{n},$$

\

which is not equal to $a_{2} \ot a_{1} \ot a_{3} \cdots \ot a_{n}$ unless one of $a_{1}$ or $a_{2}$ has even degree.

\

\hspace{3mm} In general, one may check that

$$g \cdot (a_{1} \ot \cdots \ot a_{n}) = (-1)^{Q_{g}(\deg(a_{1}), \cdots, \deg(a_{n}))} a_{g^{-1}(1)} \ot \cdots \ot a_{g^{-1}(n)},$$

\

where $Q_{g}(x_{1}, \dots, x_{n}) = \sum_{1 \leq i < j \leq n}\eps_{ij}(g) x_{i}x_{j} \in \bZ[x_{1}, \dots, x_{n}]$ is defined by

$$\eps_{ij}(g) := \left\{
	\begin{array}{ll}
	1 & \mbox{if } g(i) > g(j) \mbox{ and}  \\
	0 & \mbox{if } g(i) < g(j),
	\end{array}\right.$$

\

is the $g$-action on $H^{\bl}(X)^{\ot n}$ that is compatible with the $g$-action on $H^{\bl}(X^{n})$. This is the most crucial observation in Macdonald's work \cite{Mac1}, which we fully use for the proof of Theorem \ref{main}.

\

\subsection{General set-up} Throughout this section we fix a ground field $k$. Let $V = \bop_{i \geq 0}V_{i}$ be a graded vector space over $k$. Given $n \in \bZ_{\geq 0}$, consider the $n$-fold tensor product $V^{\ot n}$ of $V$ over $k$, where $V^{\ot 0} = k$. We have 

$$V^{\ot n} = \bop_{r \geq 0} (V^{\ot n})_{r},$$

\

where 

$$(V^{\ot n})_{r} = \bop_{i_{1} + \cdots + i_{n} = r} V_{i_{1}} \ot \cdots \ot V_{i_{n}}.$$

\

This makes $V^{\ot n}$ a graded vector space over $k$. Given any subgroup $G \leqs S_{n}$, we define the action of $G$ on $V^{\ot n}$ as follows:

$$g \cdot (v_{1} \ot \cdots \ot v_{n}) = (-1)^{Q_{g}(\deg(v_{1}), \cdots, \deg(v_{n}))} v_{g^{-1}(1)} \ot \cdots \ot v_{g^{-1}(n)},$$

\

for homogeneous $v_{1}, \dots, v_{n} \in V$ (i.e., $v_{i} \in V_{\deg(v_{i})}$), where 

$$Q_{g}(x_{1}, \dots, x_{n}) = \sum_{1 \leq i < j \leq n}\eps_{ij}(g) x_{i}x_{j} \in \bZ[x_{1}, \dots, x_{n}]$$ 

\

is defined by

$$\eps_{ij}(g) := \left\{
	\begin{array}{ll}
	1 & \mbox{if } g(i) > g(j) \mbox{ and}  \\
	0 & \mbox{if } g(i) < g(j).
	\end{array}\right.$$

\

\hspace{3mm} One may also introduce a formal cup product with the K\"unneth formula in this linear algebraic setting. That is, given any homogeneous $v, w \in V$, we define the \textbf{K\"unneth labeling} of the pure tensor $v \ot w \in V^{\ot 2}$ as the following symbol: 

$$p_{1}(v) \cup p_{2}(w) := v \ot w.$$

\

This notation does not mean anything new yet, but we also define 

$$p_{2}(w) \cup p_{1}(v) := (-1)^{\deg(v)\deg(w)} v \ot w$$

\

so that 

$$p_{2}(w) \cup p_{1}(v) = (-1)^{\deg(v)\deg(w)} p_{1}(v) \cup p_{2}(w).$$

\

On $V^{\ot n}$, the K\"unneth labeling can be recursively extend. That is, we have 

$$p_{1}(v_{1}) \cup \cdots \cup p_{n}(v_{n}) = v_{1} \ot \cdots \ot v_{n},$$

\

given homogeneous elements $v_{1}, \dots, v_{n} \in V$, and for $1 \leq i < n$, we require that swapping $p_{i}(v_{i})$ with $p_{i+1}(v_{i+1})$ must introduce the sign $(-1)^{\deg(v_{i})\deg(v_{i+1})}$. In particular, we have

$$p_{1}(v_{1}) \cup \cdots \cup p_{i+1}(v_{i+1}) \cup p_{i}(v_{i}) \cup \cdots \cup p_{n}(v_{n}) = (-1)^{\deg(v_{i})\deg(v_{i+1})}v_{1} \ot \cdots \ot v_{n},$$

\

and applying this sign rule multiple times is also allowed. One upshot is that we have

$$g \cdot (v_{1} \ot \cdots \ot v_{n}) = p_{g(1)}(v_{1}) \cup \cdots \cup p_{g(n)}(v_{n}),$$

\

where the left-hand side is defined as above, and it is easier to check some formal properties on the right-hand side. For instance, given any $\sigma, \tau \in G$, we have

\begin{align*}
\sigma(\tau(p_{1}(v_{1}) \cup \cdots \cup p_{n}(v_{n}))) &:= p_{\sigma(\tau(1))}(v_{1}) \cup \cdots \cup p_{\sigma(\tau(n))}(v_{n}) \\
&= (\sigma\tau)(p_{1}(v_{1}) \cup \cdots \cup \phi_{n}(v_{n})),
\end{align*}

\

so we have defined a set-theoretic $G$-action on a $k$-linear basis of $V^{\ot n},$ which gives rise to a $k$-linear action of $G$ on $V^{\ot n}$. We call the corresponding $k$-linear $G$-representation $G \ra \GL_{k}(V^{\ot n})$ the \textbf{K\"unneth representation} of $G$ on $V^{\ot n}$. It is important to note that the K\"unneth representation respects the grading of $V^{\ot n}$. That is, it can be thought of $G \ra \GL_{k}((V^{\ot n})_{r})$ for each $r \geq 0$.

\

\hspace{3mm} Now, to discuss traces of linear endomorphisms, assume that each homogeneous piece $V_{i}$ of $V$ is finite-dimensional. Let $\phi \in \End_{k}(V)$ be graded (with degree $0$) meaning that $\phi = \bop_{i \geq 0} \phi_{i}$, where $\phi_{i} \in \End_{k}(V_{i})$. This means that if $v \in V$ is a homogeneous element, then $\phi(v) \in V$ is a homogeneous element of degree $\deg(v)$ so that $\phi(v) = \phi_{\deg(v)}(v)$. Consider the Lefschetz series

$$L_{u}(\phi) := \sum_{i \geq 0}(-u)^{i}\Tr(\phi_{i}) \in k \llb u \rrb$$

\

of $\phi$ in $u$. It is important to note that when we have another graded endomorphism $\psi = \bop_{i \geq 0}\psi_{i}$ on $V$ and a constant $c \in k$, we have

$$L_{u}(\phi + c\psi) = L_{u}(\phi) + cL_{u}(\psi).$$

\

We also get the induced endomorphism $\phi^{\ot n} \in \End_{k}(V^{\ot n})$ given by 

$$\phi^{\ot n}(v_{1} \ot \cdots \ot v_{n}) = \phi(v_{1}) \ot \cdots \ot \phi(v_{n})$$

\

for homogeneous $v_{1}, \dots, v_{n} \in V$, which hence respects the grading of $V^{\ot n}$ so that we can write

$$\phi^{\ot n} = \bop_{r \geq 0}(\phi^{\ot n})_{r},$$

\

where

$$(\phi^{\ot n})_{r} = \bop_{i_{1} + \cdots + i_{n} = r} \phi_{i_{1}} \ot \cdots \ot \phi_{i_{n}} \in \End_{k}((V^{\ot n})_{r}) = \End_{k}\left(\bop_{i_{1} + \cdots + i_{n} = r} V_{i_{1}} \ot \cdots \ot V_{i_{n}}\right).$$

\

Given any $g \in G \leqs S_{n}$ and homogeneous $v_{1}, \dots, v_{n} \in V$, we define

\begin{align*}
(g \cdot \phi^{\ot n})(v_{1} \ot \cdots \ot v_{n}) &:= g (\phi(v_{1}) \ot \cdots \ot \phi(v_{n})) \\
&= g (p_{1}(\phi_{\deg(v_{1})}(v_{1})) \cup \cdots \cup p_{n}(\phi_{\deg(v_{n})}(v_{n}))) \\
&= p_{g(1)}(\phi_{\deg(v_{1})}(v_{1})) \cup \cdots \cup p_{g(n)}(\phi_{\deg(v_{n})}(v_{n})) \\
&= (-1)^{Q_{g}(\deg(v_{1}), \cdots, \deg(v_{n}))}\phi_{\deg(v_{g^{-1}(1)})}(v_{g^{-1}(1)}) \ot \cdots \ot \phi_{\deg(v_{g^{-1}(n)})}(v_{g^{-1}(n)}),
\end{align*}

\

where the action of $g$ on the right-hand side is given by the K\"unneth representation. This extends to a $k$-linear endomorphism $g \phi^{\ot n}$ on $V^{\ot n}$. It is important to note that we have the following commutativity:

\begin{lem}\label{comm} Keeping the notations above, we have

$$(g\phi^{\ot n})(v_{1} \ot \cdots \ot v_{n})= \phi^{\ot n}(g(v_{1} \ot \cdots \ot v_{n})).$$

\end{lem}

\

\hspace{3mm} The following is the core of the proof of Theorem \ref{main}:

\begin{thm}[Trace formula on $V^{\ot n}$]\label{main2} Assume the notations given in this section. For any $\sigma \in S_{n}$, we have

$$L_{u}(\sigma \phi^{\ot n}) = L_{u}(\phi)^{m_{1}(\sigma)}L_{u^{2}}(\phi^{2})^{m_{2}(\sigma)} \cdots L_{u^{n}}(\phi^{n})^{m_{n}(\sigma)} \in k\llb u \rrb.$$
\end{thm}

\

\begin{rmk} In the setting of Section \ref{setup}, Theorem \ref{main2} gives

$$L_{u}(g \phi_{X^{n}}) = L_{u}(\phi)^{m_{1}(g)}L_{u^{2}}(\phi^{2})^{m_{2}(g)} \cdots L_{u^{n}}(\phi^{n})^{m_{n}(g)} \in k\llb u \rrb,$$

\

for any $g \in G \leqs S_{n}$ as mentioned in the introduction. We note that until now there is no extra condition on the field $k$.
\end{rmk}

\

\hspace{3mm} In our proof of Theorem \ref{main2}, we will make use of the following properties about the quadratic forms $Q_{g}(x_{1}, \dots, x_{n})$ defined above that we learned from \cite{Mac1}. Both properties are immediate from definition:

\begin{lem}\label{properties} For any disjoint $\sigma, \tau \in S_{n}$, we have 

$$Q_{\sigma\tau}(\bs{x}) = Q_{\sigma}(\bs{x}) + Q_{\tau}(\bs{x}).$$

\

If $\sigma$ is the cycle of the form $\sigma = (\ld + 1 \ \ld + 2 \ \cdots \ \ld + r)$ with $1 \leq r \leq n$ (and $0 \leq \ld \leq n-1$), then

$$Q_{\sigma}(\bs{x}) = (x_{\ld+1} + x_{\ld+2} + \cdots + x_{\ld+r-1})x_{\ld + r}.$$
\end{lem}

\

\begin{proof}[Proof of Theorem \ref{main2}] Since the identity is only regarding traces of (homogeneous parts of) endomorphisms $\sigma \phi^{\ot n}$ and $\phi, \phi^{2}, \dots, \phi^{n},$ we may assume that $k$ is algebraically closed. Both sides of the identity are power series in $k \llb u \rrb$, so it is enough to show that for any $r \in \bZ_{\geq 0}$, their coefficients of $u^{r}$ match. This lets us reduce the problem to the case $V = V_{0} \op \cdots \op V_{r}$ and $\phi = \phi_{1} \op \cdots \op \phi_{r}$ essentially because

$$(V^{\ot n})_{r} = \bop_{i_{1} + \cdots + i_{n} = r} V_{i_{1}} \ot \cdots \ot V_{i_{r}},$$

\

where the right-hand side only consists of tensor products of $V_{0}, \dots, V_{r}$. In particular, we are now dealing with the case where $d = \dim_{k}(V) = \dim_{k}(V_{0}) + \cdots + \dim_{k}(V_{r})$ is finite.

\

\hspace{3mm} Considering $\phi \in \Mat_{d}(k) = \bA^{d^{2}}(k)$, where $d = \dim_{k}(V)$, we note that the desired equality for the coefficients of $u^{r}$ cuts out a closed subset in $\bA^{d^{2}}(k)$, with respect to the Zariski topology (on the set of closed points in $\bA^{d^{2}}$ over $k$) as we can use the Kronecker product for the matrix form of $\phi^{\ot n}$. The matrices with distinct eigenvalues form a Zariski open subset in $\Mat_{d}(k) = \bA^{d^{2}}(k)$ because we can understand them as points of the locus whose discriminant of the characteristic polynomial is nonzero. Note that we used perfectness of $k$, as now $k$ is algebriaclly closed, to ensure that any separable monic polynomial in $k[x]$ is square-free. This open locus is nonempty because $k$ has at least $d$ elements as it is infinite now that we are in the setting where $k$ is algebraically closed. Thus, such matrices are dense in $\Mat_{d}(k) = \bA^{d^{2}}(k)$, as the affine space is irreducible. This means that it is enough show the desired statement for $\phi$ with distinct eigenvalues, and this means that each $\phi_{i}$ is diagonalizable.

\

\hspace{3mm} Thus, we may find $g_{i} \in \GL_{d_{i}}(k) = \GL(V_{i})$ such that $g_{i}\phi_{i}g_{i}^{-1}$ is a diagonal matrix whose diagonal entries are eigenvalues of $\phi_{i}$, where $d_{i} = \dim_{k}(V_{i})$. Then $g_{i}\phi_{i}^{m}g_{i}^{-1}$ for any $m \geq 1$ is a diagonal matrix whose diagonal entries consists of $m$-th powers of the full list eigenvalues of $\phi_{i}$ counting with multiplicity. Writing $g = g_{1} \op \cdots \op g_{r} \in \GL_{d}(k)$, we see $gFg^{-1} = g_{1}\phi_{1}g_{1}^{-1} \op \cdots \op g_{r}\phi_{r}g_{r}^{-1}$ is a diagonal matrix, and so is 

$$(g \phi g^{-1})^{m} = g\phi^{m}g^{-1} = g_{1}\phi_{1}^{m}g_{1}^{-1} \op \cdots \op g_{r}\phi_{r}^{m}g_{r}^{-1}.$$

\

Note that $g$ respects the grading of $V$ and commutes with the action of $\sigma$:
 
\begin{align*}
g \sigma(v_{1} \ot \cdots \ot v_{n}) &= (-1)^{Q_{\sigma}(\deg(v_{1}), \cdots, \deg(v_{n}))} g (v_{\sigma^{-1}(1)} \ot \cdots \ot v_{\sigma^{-1}(n)}) \\
&= (-1)^{Q_{\sigma}(\deg(v_{1}), \cdots, \deg(v_{n}))} (g_{\deg(v_{1})} v_{\sigma^{-1}(1)}) \ot \cdots \ot (g_{\deg(v_{n})} v_{\sigma^{-1}(n)}) \\
&= \sigma ((g_{\deg(v_{1})} v_{1}) \ot \cdots \ot (g_{\deg(v_{n})} v_{n})) \\
&= \sigma g (v_{1} \ot \cdots \ot v_{n}),
\end{align*}

\

for homogeneous $v_{1}, \dots, v_{n} \in V$. Since $(g \phi g^{-1})^{\ot n} = g^{\ot n} \phi^{\ot n} (g^{-1})^{\ot n}$, we have

$$(\sigma(g \phi g^{-1})^{\ot n})_{r} =  (g^{\ot n} \sigma \phi^{\ot n} (g^{-1})^{\ot n})_{r} = (g^{\ot n})_{r} (\sigma \phi^{\ot n})_{r} ((g^{-1})^{\ot n})_{r}.$$

\

Since 

$$g^{\ot n} (g^{-1})^{\ot n}(v_{1} \ot \cdots \ot v_{n}) = gg^{-1}v_{1} \ot \cdots \ot gg^{-1}v_{n} = v_{1} \ot \cdots \ot v_{n}$$

\

for any homogeneous $v_{1}, \dots, v_{n} \in V$, we see that $(g^{\ot n})_{r}$ and $((g^{-1})^{\ot n})_{r}$ are both invertible $k$-linear endomorphisms on $(V^{\ot n})_{r}$. Thus, replacing $\phi$ with $g \phi g^{-1}$ and $\phi_{i}^{m}$ with $g_{i}\phi_{i}^{m}g_{i}^{-1}$ will not affect the desired identity, so our problem is reduced to the case where each $\phi_{i}$ is diagonal, which in particular lets us assume that  $\phi$ is diagonal.

\

\hspace{3mm} Let $v_{i,1}, \dots, v_{i,d_{i}} \in V_{i}$ be homogeneous elements forming an eigenbasis of $V_{i}$ for $\phi_{i}$ as we vary $i \geq 0$. We shall denote the corresponding eigenvalues as $\alpha_{i,j} \in k$ so that $\phi(v_{i,j}) = \phi_{i}(v_{i,j}) = \alpha_{i,j}v_{i,j}$. To compute the coefficient of $u^{r}$ on the left-hand side, fix any element 

$$w_{1} \ot \cdots \ot w_{n} \in (V^{\ot n})_{r} = \bop_{i_{1} + \cdots + i_{n} = r}V_{i_{1}} \ot \cdots \ot V_{i_{n}},$$

\

where $w_{j} = v_{i_{j}, h_{j}}$ for some $h_{j}$ so that $\deg(w_{j}) = i_{j}$ and $\phi(w_{j}) = \phi_{i_{j}}(w_{j}) = \alpha_{i_{j}, h_{j}}w_{j}$. We have

\begin{align*}
(\sigma \phi^{\ot n}) (w_{1} \ot \cdots \ot w_{n}) &= \phi^{\ot n}( \sigma (w_{1} \ot \cdots \ot w_{n}) ) \\
&= (-1)^{Q_{\sigma}(i_{1}, \dots, i_{n})} \phi(w_{\sigma^{-1}(1)}) \ot \cdots \ot \phi(w_{\sigma^{-1}(n)}) \\
&= (-1)^{Q_{\sigma}(i_{1}, \dots, i_{n})} \alpha_{i_{\sigma^{-1}(1), h_{\sigma^{-1}(1)}}}w_{\sigma^{-1}(1)} \ot \cdots \ot \alpha_{i_{\sigma^{-1}(n), h_{\sigma^{-1}(n)}}}w_{\sigma^{-1}(n)} \\
&= \alpha_{i_{\sigma^{-1}(1), h_{\sigma^{-1}(1)}}} \cdots \alpha_{i_{\sigma^{-1}(n), h_{\sigma^{-1}(n)}}} (-1)^{Q_{\sigma}(i_{1}, \dots, i_{n})} w_{\sigma^{-1}(1)} \ot \cdots \ot w_{\sigma^{-1}(n)} \\
&= \alpha_{i_{1, h_{1}}} \cdots \alpha_{i_{n}, h_{n}} (-1)^{Q_{\sigma}(i_{1}, \dots, i_{n})} w_{\sigma^{-1}(1)} \ot \cdots \ot w_{\sigma^{-1}(n)},
\end{align*}

\

so the vector $w_{1} \ot \cdots \ot w_{n}$ can possibly contribute nonzero amount to $\Tr(\tau \phi^{\ot n})_{r}$ when $w_{j} = w_{\sigma^{-1}(j)}$ for all $1 \leq j \leq n$. Now, the key is to note that the statement only depends on the cycle type of $\sigma$ in $S_{n}$ because any other $\tau \in S_{n}$ with the same cycle type is conjugate to $\sigma$ in $S_{n}$ so that $\tau = \omega \sigma \omega^{-1}$ for some $\omega \in S_{n}$ gives us

$$\Tr(\tau \phi^{\ot n})_{r} = \Tr(\omega \sigma \omega^{-1} \phi^{\ot n})_{r} = \Tr(\omega \sigma \phi^{\ot n} \omega^{-1})_{r} = \Tr(\omega_{r} (\sigma \phi^{\ot n})_{r} (\omega^{-1})_{r}) = \Tr(\sigma \phi^{\ot n})_{r}.$$

\

Thus, we have reduced the problem to the case where we have the following cycle decomposition for $\sigma$:

$$\sigma = (1 \ \cdots \ \ld_{1})(\ld_{1} + 1 \ \cdots \ld_{1} + \ld_{2}) \cdots (\ld_{1} + \cdots + \ld_{l-1} + 1 \ \cdots \ \ld_{1} + \cdots + \ld_{l}),$$

\

where $\ld_{1} + \cdots + \ld_{l} = n$. In this situation, saying that $w_{j} = w_{\sigma^{-1}(j)}$ for all $1 \leq j \leq n$ is equivalent to saying

\

\begin{itemize}
	\item $y_{1} := w_{1} = \cdots = w_{\ld_{1}}$,
	\item $y_{2} := w_{\ld_{1} + 1} = \cdots = w_{\ld_{1} + \ld_{2}}$,

\hspace{3mm} $\vdots$

	\item $y_{l} := w_{\ld_{1} + \cdots + \ld_{l-1} + 1}= \cdots = w_{\ld_{1} + \cdots + \ld_{l}}$,
\end{itemize}

\

while $y_{1}, \dots, y_{l}$ may or may not be distinct. This also shows that

\

\begin{itemize}
	\item $e_{1} := \deg(y_{1}) = i_{1} = \cdots = i_{\ld_{1}}$,
	\item $e_{2} := \deg(y_{2}) = i_{\ld_{1} + 1} = \cdots = i_{\ld_{1} + \ld_{2}}$,

\hspace{3mm} $\vdots$

	\item $e_{l} := \deg(y_{l}) = i_{1 + \cdots + \ld_{l-1} + 1} = \cdots = i_{\ld_{1} + \cdots + \ld_{l}}$.
\end{itemize}

\

Thus, we also have

\

\begin{itemize}
	\item $\alpha_{1} := \alpha_{i_{1}, h_{1}} = \cdots = \alpha_{i_{\ld_{1}}, h_{\ld_{1}}}$,
	\item $\alpha_{2} := \alpha_{i_{\ld_{1} + 1}, h_{\ld_{1} + 1}} = \cdots = \alpha_{i_{\ld_{1} + \ld_{2}}, h_{\ld_{1} + \ld_{2}}}$,

\hspace{3mm} $\vdots$

	\item $\alpha_{l} := \alpha_{i_{\ld_{1} + \cdots + \ld_{l-1} + 1}, h_{\ld_{1} + \cdots + \ld_{l-1} + 1}}= \cdots = \alpha_{i_{n}, h_{n}}$.
\end{itemize}

\

Note that $y_{j} \in V_{e_{j}}$ and $\phi(y_{j}) = \phi_{e_{j}}(y_{j}) = \alpha_{j}y_{j}$. We also note that $\ld_{1}e_{1} + \cdots + \ld_{l}e_{l} = r$ because $(V_{e_{1}})^{\ot \ld_{1}} \ot \cdots \ot (V_{e_{l}})^{\ot \ld_{l}}$ is a direct summand of $(V^{\ot n})_{r}$ in the decomposition of $V^{\ot n}$ that gives the grading for the tensor product.

\

\hspace{3mm} Thus, for this particular $\sigma$, applying Lemma \ref{properties}, we have

\begin{align*}
Q_{\sigma}(i_{1}, \dots, i_{n}) &= Q_{(1 \ \cdots \ \ld_{1})}(i_{1}, \cdots, i_{n}) + \cdots + Q_{(\ld_{1} + \cdots + \ld_{l-1} + 1 \ \cdots \ \ld_{1} + \cdots + \ld_{l})}(i_{1}, \cdots, i_{n}) \\
&= (i_{1} + \cdots + i_{\ld_{1}-1})i_{\ld_{1}} + \cdots + (i_{\ld_{1} + \cdots + \ld_{l-1} + 1} + \cdots + i_{\ld_{1} + \cdots + \ld_{l}-1})i_{\ld_{1} + \cdots + \ld_{l}}\\
&= (\ld_{1}-1)e_{1} \cdot e_{1} + \cdots + (\ld_{l}-1)e_{l} \cdot e_{l} \\
&= (\ld_{1} - 1)e_{1}^{2} + \cdots + (\ld_{l} - 1)e_{l}^{2}.
\end{align*}

\

This implies that

\begin{align*}
Q_{\sigma}(i_{1}, \dots, i_{n}) &\equiv (\ld_{1} + 1)e_{1} + \cdots + (\ld_{l} + 1)e_{l} \\
&= r + e_{1} + \cdots + e_{l},
\end{align*}

\

where the equivalence is taken modulo $2$. Hence, we have computed the sign:

$$(-1)^{Q_{\sigma}(i_{1}, \dots, i_{n})} = (-1)^{r + e_{1} + \cdots + e_{l}}.$$

\

This implies that the vector $v_{1} \ot \cdots \ot v_{n} = y_{1}^{\ot \ld_{1}} \ot \cdots y_{l}^{\ot \ld_{l}}$ contributes

\begin{align*}
(-1)^{r + e_{1} + \cdots + e_{l}} \alpha_{i_{1}, h_{1}} \cdots \alpha_{i_{n}, h_{n}} &= (-1)^{r + e_{1} + \cdots + e_{l}} \alpha_{1}^{\ld_{1}} \cdots \alpha_{l}^{\ld_{l}}
\end{align*}

\

to $\Tr(\sigma \phi^{\ot n})_{r}$. Keeping the partition $[\ld_{1}, \dots, \ld_{l}] \vdash n$, which is the equivalent datum to the cycle decomposition of $\sigma$, we have

$$\Tr(\sigma \phi^{\ot n})_{r} = \sum_{\ld_{1}e_{1} + \cdots + \ld_{l}e_{l} = r} \sum_{\substack{(y_{1}, \dots, y_{n}) \in V_{e_{1}} \times \cdots \times V_{e_{l}} \\ \text{ basis elements of } V_{e_{i}}}} (-1)^{r + e_{1} + \cdots + e_{l}} \alpha_{1}^{\ld_{1}} \cdots \alpha_{l}^{\ld_{l}}.$$

\

This implies that

\begin{align*}
L_{u}(\sigma \phi^{\ot n}) &= \sum_{r \geq 0} (-u)^{r}\Tr(\sigma \phi^{\ot n})_{r} \\
&= \sum_{r \geq 0}  \sum_{\ld_{1}e_{1} + \cdots + \ld_{l}e_{l} = r} \sum_{\substack{(y_{1}, \dots, y_{n}) \in V_{e_{1}} \times \cdots \times V_{e_{l}} \\ \text{ basis elements of } V_{e_{i}}}} (-1)^{e_{1} + \cdots + e_{l}} \alpha_{1}^{\ld_{1}} \cdots \alpha_{l}^{\ld_{l}} u^{r} \\
&= \sum_{r \geq 0}  \sum_{\ld_{1}e_{1} + \cdots + \ld_{l}e_{l} = r} \sum_{\substack{(y_{1}, \dots, y_{n}) \in V_{e_{1}} \times \cdots \times V_{e_{l}} \\ \text{ basis elements of } V_{e_{i}}}} (-1)^{e_{1} + \cdots + e_{l}} \alpha_{1}^{\ld_{1}} \cdots \alpha_{l}^{\ld_{l}} u^{\ld_{1}e_{1} + \cdots + \ld_{l}e_{l}} \\
&= \sum_{r \geq 0}  \sum_{\ld_{1}e_{1} + \cdots + \ld_{l}e_{l} = r} \sum_{\substack{(y_{1}, \dots, y_{n}) \in V_{e_{1}} \times \cdots \times V_{e_{l}} \\ \text{ basis elements of } V_{e_{i}}}} \alpha_{1}^{\ld_{1}} \cdots \alpha_{l}^{\ld_{l}} (-u^{\ld_{1}})^{e_{1}} \cdots (-u^{\ld_{l}})^{e_{l}} \\
&= \sum_{r \geq 0}  \sum_{\ld_{1}e_{1} + \cdots + \ld_{l}e_{l} = r} \left(\sum_{\substack{y_{1} \in V_{e_{1}} \\ \text{ basis elements}}} \alpha_{1}^{\ld_{1}} (-u^{\ld_{1}})^{e_{1}}\right) \cdots \left(\sum_{\substack{y_{l} \in V_{e_{l}} \\ \text{ basis elements}}} \alpha_{l}^{\ld_{l}} (-u^{\ld_{l}})^{e_{l}}\right) \\
&= \sum_{r \geq 0}  \sum_{\ld_{1}e_{1} + \cdots + \ld_{l}e_{l} = r} \Tr(\phi_{e_{1}}^{\ld_{1}}) (-u^{\ld_{1}})^{e_{1}} \cdots \Tr(\phi_{e_{l}}^{\ld_{l}}) (-u^{\ld_{l}})^{e_{l}} \\
&= \sum_{e_{1}, \dots e_{l} \geq 0} \Tr(\phi_{e_{1}}^{\ld_{1}}) (-u^{\ld_{1}})^{e_{1}} \cdots \Tr(\phi_{e_{l}}^{\ld_{l}}) (-u^{\ld_{l}})^{e_{l}} \\
&= \left(\sum_{e_{1} \geq 0} \Tr(\phi_{e_{1}}^{\ld_{1}}) (-u^{\ld_{1}})^{e_{1}}\right) \cdots \left(\sum_{e_{l} \geq 0} \Tr(\phi_{e_{l}}^{\ld_{l}}) (-u^{\ld_{l}})^{e_{l}}\right) \\
&= \left(\sum_{i \geq 0} \Tr(\phi_{i}^{\ld_{1}}) (-u^{\ld_{1}})^{i}\right) \cdots \left(\sum_{i \geq 0} \Tr(\phi_{i}^{\ld_{l}}) (-u^{\ld_{l}})^{i}\right) \\
&= \left(\sum_{i \geq 0} \Tr(\phi_{i}) (-u)^{i}\right)^{m_{1}(\sigma)} \left(\sum_{i \geq 0} \Tr(\phi_{i}^{2}) (-u^{2})^{i}\right)^{m_{2}(\sigma)} \cdots \left(\sum_{i \geq 0} \Tr(\phi_{i}^{n}) (-u^{n})^{i}\right)^{m_{n}(\sigma)} \\
&= L_{u}(\phi)^{m_{1}(\sigma)}L_{u^{2}}(\phi^{2})^{m_{2}(\sigma)} \cdots L_{u^{n}}(\phi^{n})^{m_{n}(\sigma)}.
\end{align*}

\

as desired, where we note that $\alpha_{i}$ appearing in the computation above depends on the choice of $y_{i} \in V_{e_{i}}$.
\end{proof}

\

\subsection{Proof of Theorem \ref{main}} Keeping all the notations in the previous subsection, the following immediately proves Theorem \ref{main}:

\begin{thm}[Trace formula on $(V^{\ot n})^{G}$]\label{key} Let $G \leqslant S_{n}$ such that $|G| \neq 0$ in $k$. Then

$$L_{u}(\phi^{\ot n} |_{ (V^{\ot n})^{G} }) = \frac{1}{|G|}\sum_{g \in G}L_{u}(\phi)^{m_{1}(g)}L_{u^{2}}(\phi^{2})^{m_{2}(g)} \cdots L_{u^{n}}(\phi^{n})^{m_{n}(g)}.$$
\end{thm}

\begin{proof} Since $|G| \neq 0$ in $k$, we can consider the averaging operator $e_{G} : V^{\ot n} \rightarrow V^{\ot n}$ given by

$$e_{G}(\alpha) := \frac{1}{|G|}\sum_{g \in G} g\alpha,$$

\

where again, we use the K\"unneth representation of $G$ on $V^{\ot n}$ introduced in the beginning of this section. Note that we have $(V^{\ot n})^{G} = e_{G}(V^{\ot n})$, so any element of $(V^{\ot n})^{G}$ can be written as $e_{G}(\alpha)$ with $\alpha \in V^{\ot n}$. Using Lemma \ref{comm}, we have

$$\phi^{\ot n}(e_{G}(\alpha)) = \phi^{\ot n}\left(\frac{1}{|G|}\sum_{\sigma \in G}\sigma\alpha\right) = \frac{1}{|G|}\sum_{\sigma \in G}\sigma \phi^{\ot n}(\alpha).$$

\

Thus, we have shown that 

$$\phi^{\ot n} \circ e_{G} = \frac{1}{|G|}\sum_{\sigma \in G}\sigma \phi^{\ot n} \in \End_{k}(V^{\ot n}).$$

\

Note that both sides restrict to $(V^{\ot n})^{G}$, and since $e_{G}$ is the identity on $(V^{\ot n})^{G}$, we get

$$\phi^{\ot n}|_{(V^{\ot n})^{G}} = \frac{1}{|G|}\sum_{\sigma \in G}\sigma \phi^{\ot n} \in \End_{k}((V^{\ot n})^{G}).$$

\

Applying $L_{u}$ both sides, we get

$$L_{u}(\phi^{\ot n}|_{(V^{\ot n})^{G}}) = \frac{1}{|G|}\sum_{\sigma \in G}L_{u}(\sigma \phi^{\ot n}),$$

\

so applying Theorem \ref{main2}, we are done.
\end{proof}

\

\section{Alternating powers}\label{alt}

\hspace{3mm} In the introduction, say for Corollary \ref{S_{n}}, we only cared about full symmetric groups $(S_{n})_{n \in \bZ_{\geq 0}}$. It is natural to consider other sequence of subgroups of $S_{n}$ for $n \in \bZ_{\geq 0}$. In this section, we consider alternating groups, using a well-known lemma in combinatorics:

\begin{lem}[p.36 of \cite{HN}]\label{comb1} For any $n \in \bZ_{\geq 2}$, we have the following identity relating cycle indices of $A_{n}$ and $S_{n}$:

$$Z_{A_{n}}(x_{1}, x_{2}, \dots, x_{n}) = Z_{S_{n}}(x_{1}, x_{2} \dots, x_{n}) + Z_{S_{n}}(x_{1}, -x_{2}, \dots, (-1)^{n+1}x_{n}).$$
\end{lem}

\

\hspace{3mm} Given a sequence $G_{n} \leqs S_{n}$ of subgroups for $n \in \bZ_{\geq 0}$, we write

$$Z_{G_{\bl}}(\bs{x}, t) := \sum_{n=0}^{\infty}Z_{G_{n}}(\bs{x})t^{n} \in \bQ[\bs{x}]\llb t \rrb.$$

\

\begin{cor}\label{comb2} We have
$$Z_{A_{\bl}}(\bs{x}, t) = Z_{S_{\bl}}(\bs{x}, t) + \frac{1}{Z_{S_{\bl}}(\bs{x}, -t)} - 1 - x_{1}.$$
\end{cor}

\begin{proof} Recall from the introduction that

$$\sum_{n=0}^{\infty}Z_{S_{n}}(x_{1}, \dots, x_{n}) t^{n} = \exp\lt(\sum_{r=1}^{\infty}\frac{x_{r}}{r}t^{r}\rt),$$

\

which implies that

\begin{align*}
\sum_{n=0}^{\infty}Z_{S_{n}}(x_{1}, -x_{2} \dots, (-1)^{n+1}x_{n}) t^{n} &= \exp\lt(\sum_{r=1}^{\infty}\frac{(-1)^{r+1}x_{r}}{r}t^{r}\rt) \\
&= \exp\lt(-\sum_{r=1}^{\infty}\frac{x_{r}}{r}(-t)^{r}\rt) \\
&= \exp\lt(\sum_{r=1}^{\infty}\frac{x_{r}}{r}(-t)^{r}\rt)^{-1} \\
&= Z_{S_{\bl}}(\bs{x}, -t)^{-1}.
\end{align*}

\

Therefore, applying Lemma \ref{comb1}, we are done.
\end{proof}

\

\hspace{3mm} Applying Lemma \ref{comb1} and Corollary \ref{comb2}, we have the following:

\begin{cor}\label{A_{n}} Assume the same hypotheses as in Theorem \ref{main}. Then we have

$$\sum_{n=0}^{\infty}L_{u}(\phi_{X^{n}}|_{H^{\bl}(X^{n})^{A_{n}}}) t^{n} = \sum_{n=0}^{\infty}L_{u}(\phi_{X^{n}}|_{H^{\bl}(X^{n})^{S_{n}}}) t^{n} + \frac{1}{\sum_{n=0}^{\infty}L_{u}(\phi_{X^{n}}|_{H^{\bl}(X^{n})^{S_{n}}}) (-t)^{n}} - 1 - L_{u}(\phi).$$

\

If $\dim_{k}(H^{\bl}(X))$ is finite so that $H^{i}(X) = 0$ for all $i > 2d$ for some $d$, then

$$\sum_{n=0}^{\infty}L_{u}(\phi_{X^{n}}|_{H^{\bl}(X^{n})^{A_{n}}}) t^{n} = \prod_{i=0}^{2d} \left( \frac{1}{\det(\id_{H^{i}(X)} - \phi_{i}u^{i}t)} \right)^{(-1)^{i}} +  \prod_{i=0}^{2d} \left( \frac{1}{\det(\id_{H^{i}(X)} + \phi_{i}u^{i}t)} \right)^{(-1)^{i+1}} - 1 - L_{u}(\phi).$$

\end{cor}

\

\

\hspace{3mm} Just as Corollary \ref{S_{n}} implied Theorem \ref{showcase}, Corollary \ref{A_{n}} implies the following concrete theorem:

\begin{thm}\label{showcase2} Let $X$ be either a compact complex manifold of dimension $d$ or a projective variety of dimension $d$ over a finite field $\bF_{q}$. Then for any endomorphism $F$ on $X$, we have

$$\sum_{n=0}^{\infty}L_{u}(\Alt^{n}(F)^{*}) t^{n} = \prod_{i=0}^{2d} \left( \frac{1}{\det(\id_{H^{i}(X)} - F^{*}_{i}u^{i}t)} \right)^{(-1)^{i}} +  \prod_{i=0}^{2d} \left( \frac{1}{\det(\id_{H^{i}(X)} + F^{*}_{i}u^{i}t)} \right)^{(-1)^{i+1}} - 1 - L_{u}(F^{*}),$$

\

where we used the same notations as in Theorem \ref{showcase} except $\Alt^{n}(F)$, the endomorphism on the $n$-th alternating power $\Alt^{n}(X) = X^{n}/A_{n}$ of $X$ induced by $F$.
\end{thm}

\

\hspace{3mm} Theorem \ref{showcase2} is interesting in its own right. For instance, in the singular setting, if we take $F = \id_{X}$, we have the following identity that computes the singular Betti numbers of $\Alt^{n}(X)$ in terms of those of $X$:

$$\sum_{n=0}^{\infty}\chi_{u}(\Alt^{n}(X)) t^{n} = \prod_{i=0}^{2d} \left( \frac{1}{1 - u^{i}t} \right)^{(-1)^{i}h^{i}(X)} +  \prod_{i=0}^{2d} \left( \frac{1}{1 + u^{i}t} \right)^{(-1)^{i+1}h^{i}(X)} - 1 - \chi_{u}(X).$$

\

In particular, taking $u = 1$, we have the following formula for the Euler characteristics:

$$\sum_{n=0}^{\infty}\chi(\Alt^{n}(X)) t^{n} = \prod_{i=0}^{2d} \left( \frac{1}{1 - t} \right)^{\chi(X)} +  \prod_{i=0}^{2d} \left( \frac{1}{1 + t} \right)^{-\chi(X)} - 1 - \chi(X).$$

\

If $X$ is a smooth projective variety over $\bC$, then we also get the alternating power analogue of Cheah's result:

\begin{align*}
\sum_{n=0}^{\infty} & \sum_{i \geq 0}\sum_{p+q=i}h^{p,q}(\Alt^{n}(X))x^{p}y^{q}(-u)^{i} t^{n} \\
&= \prod_{i = 0}^{2d} \prod_{p + q = i} \left( \frac{1}{1 - x^{p}y^{q}u^{i} t} \right)^{(-1)^{i}h^{p,q}(X)} + \prod_{i = 0}^{2d} \prod_{p + q = i} \left( \frac{1}{1 + x^{p}y^{q}u^{i} t} \right)^{(-1)^{i+1}h^{p,q}(X)} \\
&\hspace{3mm} - 1 - \sum_{i = 0}^{2d}\sum_{p+q=i}h^{p,q}(X)x^{p}y^{q}(-u)^{i} t.
\end{align*}

\

In particular, the right-hand side is rational in $t$. We can also obtain an alternating analogue of Kapranov's observation that

$$Z_{X}(t) = \sum_{n=0}^{\infty} |\Sym^{n}(X)(\bF_{q})|t^{n},$$

\

where $Z_{X}(t)$ is the zeta series of $X$. That is, taking $u = 1$ in the $l$-adic setting of Theorem \ref{showcase2}, we have

$$Z_{X}(t) + \frac{1}{Z_{X}(-t)} - 1 - |X(\bF_{q})|t = \sum_{n=0}^{\infty}|\Alt^{n}(X)(\bF_{q})|t^{n}.$$

\

\begin{rmk} In a collaboration with Yinan Nancy Wang, we started to question if this identity holds in the Grothendieck ring of varieties over a field, where taking $\bF_{q}$-point countings is replaced by taking the classes in the Grothendieck ring. The problem seems nontrivial even when $X$ is $\bA^{1}$ or $\bP^{1}$ over any field.
\end{rmk}

\

\hspace{3mm} The upshot of this section is that because Theorem \ref{main} is formulated in terms of cycle indices, we can use combinatorial knowledge about them to understand cohomological information about alternating powers. We believe that there are more sequences of subgroups $G_{n}$ of $S_{n}$ such that the generating function for certain cohomological information (e.g., singular Betti numbers, $\bF_{q}$-point counts, or Hodge numbers) of $X^{n}/G_{n}$ is rational. Namely, whenever the generating function for $Z_{G_{n}}(x_{1}, \dots, x_{n})$ has a formula that involves exponentiation, we should be able to get rationality for cohomological information of $X^{n}/G_{n}$ by applying Theorem \ref{main}.

\

\section{More on point counting over finite fields}\label{pointcount}

\hspace{3mm} Let $X$ be a projective variety over a finite field $\bF_{q}$, and consider any subgroup $G \leqs S_{n}$ acting on $X^{n}$ by permuting coordinates. An immediate consequence of Theorem \ref{main} in the $l$-adic setting, for a prime $l$ not dividing $q$ nor $|G|$, by taking $u = 1$ and applying the Grothendieck-Lefschetz trace formula is that

\begin{align*}
|(X^{n}/G)(\bF_{q})| &= Z_{G}(|X(\bF_{q})|, |X(\bF_{q^{2}})|, \dots, |X(\bF_{q^{n}})|) \\
&= \frac{1}{|G|}\sum_{g \in G}|X(\bF_{q})|^{m_{1}(g)}|X(\bF_{q^{2}})|^{m_{2}(g)} \cdots |X(\bF_{q^{n}})|^{m_{n}(g)}.
\end{align*}

\

\hspace{3mm} It turns out that the formula even holds when $X$ is a quasi-projective variety over $\bF_{q}$ by using Theorem \ref{main} with the compactly supported $l$-adic \'etale cohomology, noting that all the results we use for the $l$-adic \'etale cohomology when $X$ is projective over $\bF_{q}$ generalize to the compactly supported $l$-adic \'etale cohomology when $X$ is quasi-projective over $\bF_{q}$ as long as $l \nmid q, |G|$. In particular, taking $X = \bA^{1}$ over $\bF_{q}$, we have

$$|(\bA^{n}/G)(\bF_{q})| = q^{n}.$$

\

When $G = A_{n}$ and $q$ is odd, we have

$$\bA^{n}/A_{n} \simeq \mathrm{Spec}\left( \frac{\bF_{q}[t_{1}, \dots, t_{n}]}{(y^{2} - \Delta_{n}(t_{1}, \dots, t_{n}))} \right),$$

\

where $\Delta_{n}(t_{1}, \dots, t_{n})$ is the discriminant of the monic polynomial

$$x^{n} + t_{1}x^{n-1} + \cdots + t_{n-1}x + t_{n}.$$

\

Thus, we see that for $n \geq 2$, the polynomial function $\Delta_{n} : \bF_{q}^{n} \rightarrow \bF_{q}$ given by the discriminant satisfies

$$|\Delta_{n}^{-1}(\{\text{quadratic residues in } \bF_{q}^{\times}\})| = |\Delta_{n}^{-1}(\{\text{quadratic non-residues in } \bF_{q}^{\times}\})|,$$

\

because $|\Delta_{n}^{-1}(0)| = q^{n-1}$, as there are precisely $q^{n} - q^{n-1}$ degree $n$ monic square-free polynomials in $\bF_{q}[x]$. The above equality was also observed by Chan, Kwon, and Seaman using more direct computations (Corollary 3.3 of \cite{CKS}).

\

\section{Further directions}

\hspace{3mm} Yinan Nancy Wang conjectured that the author's use of Theorem \ref{main} to restore Cheah's result about the Hodge numbers of $X^{n}/G$ can vastly generalize, which is the content of a joint work in progress. For instance, given any endomorphism on a smooth projective variety $X$ over $\bC$, it is believable that the methods in this paper will allow us to compute the Lefschetz series of the induced map on the cohomology of the sheaf of holomorphic $p$-forms on $X^{n}/G$. Since the quotient $X^{n}/G$ is not smooth in general, discussing holomorphic forms gets more technical. Nevertheless, we conjecture that in general, one can ``unwind'' the problem of taking intersections in $X^{n}/G$ into taking intersections in $X$. It will be also interesting to investigate if there is an analogous story to this paper for other finite group quotients.

\


\end{document}